\documentclass{amsart}
\usepackage{amssymb}
\usepackage{amscd}
\usepackage{graphicx}
\usepackage{amsmath}
\usepackage{setspace}
\usepackage[all]{xy}
\usepackage{color}
\usepackage{enumerate}

\newtheorem{theorem}{Theorem}[section]

\newtheorem{corollary}[theorem]{Corollary}

\newtheorem{lemma}[theorem]{Lemma}

\newtheorem{proposition}[theorem]{Proposition}

\newtheorem{claim}{Claim}

\theoremstyle{remark}

\newtheorem{example}[theorem]{Example}
\newtheorem{remark}[theorem]{Remark}

\newcommand{\CA}{\mathcal{A}}
\newcommand{\CB}{\mathcal{B}}

\newcommand{\CALD}{\mathcal{D}}
\newcommand{\CF}{\mathcal{F}}
\newcommand{\CG}{\mathcal{G}}

\newcommand{\CP}{\mathcal{P}}

\newcommand{\CT}{\mathcal{T}}

\newcommand{\CX}{\mathcal{X}}
\newcommand{\CY}{\mathcal{Y}}

\newcommand{\frakp}{\mathfrak{p}}
\newcommand{\frakq}{\mathfrak{q}}

\newcommand{\BBQ}{{\mathbb Q}}

\newcommand{\BBZ}{{\mathbb Z}}

\newcommand{\BD}{{\mathbf D}}

\newcommand{\Hom}{\mathrm{Hom}}
\newcommand{\Ima}{\mathrm{Im}}

\newcommand{\Ker}{\mathrm{Ker}}
\newcommand{\Coker}{\mathrm{Coker}}
\newcommand{\Ext}{\mathrm{Ext}}

\newcommand{\Spec}{\mathrm{Spec}}

\newcommand{\supp}{\mathrm{ supp}}
\newcommand{\Ann}{\mathrm{ Ann}}

\newcommand{\Gen}{\mathrm{Gen}}

\newcommand{\Def}{\mathrm{Def}}

\newcommand{\Add}{\mathrm{ Add}}

\newcommand{\Modr}{\mathrm{ Mod}\text{-}}

\begin{document}


\title{The ascent-descent property for $2$-term silting complexes}

\author{Simion Breaz}

\address{Simion Breaz: "Babe\c s-Bolyai" University, Faculty of Mathematics and Computer Science, Str. Mihail Kog\u alniceanu 1, 400084, Cluj-Napoca, Romania}

\email{bodo@math.ubbcluj.ro}




\begin{abstract}
We will prove that over commutative rings the silting property of $2$-term complexes induced by morphisms between projective modules is preserved and reflected by faithfully flat extensions.
\end{abstract}

\maketitle

\section{Introduction}
Let $\lambda:R\to S$ be a homomorphism of unital rings. We consider the extension of scalars functor $-\otimes_R S:\Modr R\to \Modr S$, and we will say that a property $\CP$ associated to a complex of modules \textsl{ascends} along $\lambda$ if the functor $-\otimes_R S$ preserves the property $\CP$, i.e. for every complex $\mathbf{C}$ of right $R$-modules which satisfies $\CP$, the complex $\mathbf{C}\otimes_R S$ satisfies $\CP$ in $\Modr S$. The property $\CP$ \textsl{descends} along $\lambda$ if a complex $\mathbf{C}$ in $\Modr R$ satisfies $\CP$ provided that $\mathbf{C}\otimes_R S$ satisfies $\CP$ as a complex of right $S$-modules. The above definitions are natural extension of the corresponding ascent/descent notions associated to module properties, e.g. \cite[Definition 3.5]{EGT}. The properties of modules which ascend along flat ring homomorphisms and descend along faithful flat ring homomorphisms (called \textsl{ascent-descent} properties) play an important role in commutative algebra since the corresponding properties associated to quasi-coherent sheaves have a local character, \cite[Lemma 3.4]{EGT}. For instance, for modules over commutative rings the properties ``projective'', \cite{Gr-Ra} and \cite[Section 058B]{staks-04VM}, and ``1-tilting'', \cite[Theorem 3.13]{Hr-St-Tr}, are ascent-descent. We mention here that the ascending property of tilting also plays an important role in the non-commutative case since they are used to characterize derived equivalences, \cite{Ri}. We refer to \cite{Mi} for a general approach of this case.    

In this paper we will study the ascent and descent properties for \textsl{$2$-term silting complexes}. These are complexes $\cdots \to 0\to P^{-1}\overset{\sigma}\to P^0\to 0\to\cdots$, concentrated in $-1$ and $0$, which are silting objects in the unbounded derived category $\BD(R)$ of $R$. In order to simplify the presentation, we will identify, as in \cite{AMV:2015}, every $2$-term complex $\cdots \to 0\to P^{-1}\overset{\sigma}\to P^0\to 0\to\cdots$  with the homomorphism $\sigma:P^{-1}\to P^0$. If $\sigma$ is a  $2$-term silting complex we will say that  $\Coker(\sigma)$ is a silting module (with respect to the homomorphism $\sigma$).
These notions were introduced in \cite{AMV:2015}  as non-compact versions of the $\tau$-tilting modules, \cite{AIR}, and of the two term silting complexed, \cite{Ke-Vo}. The role of silting modules in the study of module categories is described in \cite{An-abun}. For the finitely presented case we mention here the results proved in \cite{AIR} and \cite[Section 5]{DF}. If $R$ is a ring then (bounded) silting complexes play in the derived category $\BD(R)$ a similar role with that of tilting modules in module categories, \cite{Wei-isr}. In spite of this,   correspondences or similarities between the influences of tilting modules and silting homomorphisms on the module category can be established only for particular instances (e.g. \cite{Br-Ze-2018}). We refer to \cite{An-18} for the general theory of silting objects in triangulated categories. 

In Theorem \ref{thm-1} we provide, for the general case, a characterization for the ascending property of $2$-term silting complexes. For the commutative case the ascending property of $2$-term silting complexes is valid along all ring homomorphisms, Theorem \ref{commutative-1}. Moreover, we will prove in Theorem \ref{silting} that the silting property associated to $2$-term complexes descends along faithfully flat ring homomorphisms of commutative rings.    

In this paper all rings and all ring homomorphisms are unital. If $\lambda:R\to S$ is a homomorphism of rings then the extension of scalars functor is denoted by  $-\otimes_R S:\Modr R\to \Modr S$. The restriction of scalars functors is denoted by $\lambda^*=\Hom_S(S,-):\Modr S\to \Modr R$.  When there is no danger of confusion, we will consider every class of objects in $\Modr S$ as a subclass of $\Modr R$. Therefore, if $\CX$ is a class of $R$-modules, and $\CY$ is a class of $S$-modules then $\CX\cap\CY$ means the class of all modules from $N\in\CY$ such that $\lambda^\star(N)\in \CX$. In particular, we will identify a right $S$-module $N$ with its image $\lambda^\star(N)$.  

\section{The ascent property}

Let $R$ be a unital ring. We consider a homomorphism of projective right $R$-modules  
$\sigma:P^{-1} \to P^{0}$, and we denote by $T$ the cokernel of $\sigma$.   

Then we can associate to $\sigma$ the class  
$$\CALD_{\sigma}= \{ X\in\Modr R \mid \Hom_{R}(\sigma,X) \text{ is an epimorphism}\}.$$ Since $\CALD_\sigma$ is the kernel of the functor $\Coker(\Hom_R(\sigma,-))$, by using the properties of this functor (e.g. \cite[Proposition 4]{Br_Ze:2014}), it follows that $\CALD_\sigma$ is closed with respect to epimorphic images, extensions, and direct products. 

Following \cite{AMV:2015}, we will say that $\sigma$ is \textsl{partial silting} or that the module $T=\Coker(\sigma)$ is \textsl{partial silting with respect to $\sigma$} if  
$\CALD_\sigma$ is a torsion class (i.e. $\CALD_\sigma$ is also closed under direct sums) and $T\in \CALD_\sigma$. Then $\Gen(T)\subseteq \CALD_\sigma\subseteq {T^\perp}$ and $(\Gen(T), T^{\circ})$ is a torsion pair, where $$\Gen(T)=\{M\in \Modr R\mid \textrm{there exists an epimorphism } T^{(I)}\to M\}$$ is the class of all $T$-generated right $R$-modules, ${T^\perp}=\Ker\Ext_R^1(T,-)$, and ${T^\circ}=\Ker\Hom_R(T,-)$.
If  $\CALD_\sigma=\Gen(T)$ then we will say that $\sigma$ is \textsl{a $2$-term silting complex} and $T$ is called a \textsl{silting module with respect to $\sigma$} (we recall that in \cite[Theorem 4.9]{AMV:2015} it is proved that a homomorphism $\sigma$ satisfies the above condition if and only if it represents a silting complex in the associated unbounded derived category).

If $\lambda:R\to S$ is a ring homomorphism and $\sigma:P^{-1} \to P^{0}$ is a homomorphism of projective right $R$-modules then we will denote $\sigma\otimes_R S=\sigma\otimes_R 1_S$ the induced homomorphism of projective right $S$-modules.

\begin{lemma}\label{basic1}
Let $\lambda:R\to S$ be a ring homomorphism. If $\sigma:P^{-1} \to P^{0}$ is a homomorphism between projective right $R$-modules then  $\CALD_{\sigma\otimes_R S}=\CALD_{\sigma}\cap \Modr S$.
\end{lemma}

\begin{proof}
Let $M$ be a right $S$-module. We have the commutative diagram
$$\xymatrix{\Hom_R(P^0,\Hom_S(S,M))\ar[rrr]^{\Hom_R(\sigma,\Hom_S(S,M))}\ar[d]^\cong &&&\Hom_R(P^{-1},\Hom_S(S,M))\ar[d]^{\cong}\\
\Hom_S(P^0\otimes_R S,M)\ar[rrr]^{\Hom_S(\sigma\otimes_R S,M)} &&& \Hom_S(P^{-1}\otimes_R S,M)}.$$
It follows that a right $S$-module $M$ belongs to $\CALD_{\sigma\otimes_R S}$ if and only if the right $R$-module $\lambda^\star(M)=\Hom_S(S,M)$ is in $\CALD_\sigma$.
\end{proof}

\begin{theorem}\label{thm-1}
Suppose that $\sigma$ represents a $2$-term silting complex. If $\lambda:R\to S$ is a ring homomorphism then the following are equivalent: 
\begin{enumerate}[{\rm (1)}]
 \item $\sigma\otimes_R S$ is a $2$-term silting complex in $\Modr S$;
 \item $\lambda^\star(T\otimes_R S)$ is $T$-generated.
\end{enumerate}
\end{theorem}

\begin{proof} 
(1)$\Rightarrow$(2) Since $\sigma\otimes_R S$ represents a $2$-term silting complex, it follows that $T\otimes_RS\in\CALD_{\sigma\otimes_RS}$. Then (2) is a consequence of Lemma \ref{basic1}.

(2)$\Rightarrow$(1) 
Since $\CALD_\sigma$ is closed under direct sums, it is obvious, by Lemma \ref{basic1}, that $\CALD_{\sigma\otimes_R S}$ has the same property. 

For every $M\in \CALD_{\sigma\otimes_R S}$ we have $\lambda^\star(M)\in \CALD_\sigma$. Since $\sigma$ is a $2$-term silting complex, it follows that $\lambda^\star(M)$  is $T$-generated. Then there exists an $R$-epimorphism $T^{(I)}\to \lambda^\star(M)$ which induces an $S$-epimorphism $T\otimes_R S^{(I)}\to \lambda^\star(M)\otimes_R S$. But it is well known that the canonical $S$-homomorphism $\lambda^\star(M)\otimes_R S\to M$, $x\otimes s\mapsto x(s)$, is an epimorphism. It follows that 
$\CALD_{\sigma\otimes_R S}\subseteq \Gen(T\otimes_R S)$. 

Therefore, it is enough to prove that $T\otimes_R S\in \CALD_{\sigma\otimes_R S}$. By Lemma \ref{basic1}, this is equivalent to $\lambda^\star(T\otimes_R S)\in\CALD_\sigma$. Since $\sigma$ is silting, we have $\lambda^\star(T\otimes_R S)\in\CALD_\sigma$ if and only if $\lambda^\star(T\otimes_R S)$ is $T$-generated, and this last property is assumed in (2).  
\end{proof}


\begin{remark}
A similar proof, using the isomorphism $\Hom_S(X,\Hom_R(S,I))\cong \Hom_R(X\otimes_S S,I)$, can be used to obtain the dual of Theorem \ref{thm-1} for cosilting modules (we refer to \cite{Br-Po} and \cite{ZW2} for the basic properties of these modules). Therefore, if $\lambda:R\to S$ is a unital ring homomorphism and $C\in \Modr R$ is a  cosilting module with respect to the injective copresentation $\beta:I^0\to I^1$ then the coinduced $S$-module $\Hom_R(S,C)$ is a cosilting module with respect to $\Hom_R(S,\beta)$ if and only if $\Hom_R(S,C)$ is $C$-cogenerated.
\end{remark}

In the case of surjective ring homomorphisms, the extension of scalars functor always preserves the $2$-term silting complexes.

\begin{corollary}
Let $\lambda:R\to S$ be a surjective ring homomorphism. Then for every $2$-term silting complex $\sigma$ the induced complex $\sigma\otimes_R S$ is a $2$-term silting complex. 
\end{corollary}

\begin{proof}
If $I=\Ker(\lambda)$ then for every right $R$-module $M$ we have a natural $R$-isomorphism $M\otimes_R S\cong M/IM$ (see \cite[12.11]{Wis}). 
\end{proof}

However, even in the case of non-surjective ring epimorphisms the above corollary is not true. The following example was communicated to me by Lidia Angeleri-H\"ugel.

\begin{example}
Let $R$ be the (Kronecker) path algebra associated to the graph $\mathbf{2} \leftleftarrows \mathbf{1}$ over a field $K$. The silting modules over Kronecker algebras are described in \cite[Examples 5.10, and 5.18]{AMV2} and in \cite[Example 2.20]{Jasso}. If $\mathbf{1}$ is the simple injective $R$-module and $\mathbf{2}$ is the simple projective module in $\Modr R$, we denote by $P$ an indecomposable preprojective  module which is not isomorphic to $\mathbf{2}$. By \cite[Example 5.18]{AMV2} there exists a universal localization $\lambda:R\to S$ such that $\lambda^\star(\Modr S)=\Add(P)$. Since $\lambda$ is a universal localization, the natural homomorphism $\textbf{2}\otimes_R\lambda:\textbf{2}\to \textbf{2}\otimes_R S$ is a the $\Add(P)$-reflection of $\textbf{2}$, i.e. for every $X\in \Add(P)$, every homomorphism $\textbf{2}\to X$ factorizes through $\textbf{2}\otimes_R\lambda$. But $\Hom_R(\textbf{2},\Add(P))\neq 0$, so it follows that $\textbf{2}\otimes_R S\neq 0$. Moreover, it is obvious that $\textbf{2}$ cannot generate the modules from $\Add(P)$. 
Note that $\textbf{2}$ is a silting module with respect to a homomorphism $\sigma$. By the above remarks it follows that the complex $\sigma\otimes_R S$ is not silting.           
\end{example}

In the following we will see an example of a tilting module such that the induced module with respect to a ring homomorphism is silting, but it is not $1$-tilting. This is based on the example presented in \cite[Example 4.2]{As-Ma}.

\begin{example}
Let $R$ be the algebra associated to the quiver $$\mathbf{1}\overset{\beta}\longleftarrow \mathbf{2}\overset{\alpha}\longleftarrow \mathbf{3},$$ and $S$ algebra associated to the quiver 
$$\xymatrix{\mathbf{1}\ar@{<-}[r]^{\beta} \ar@/_1pc/[rr]_{\gamma} &\mathbf{2}\ar@{<-}[r]^{\alpha}  & \mathbf{3}
}$$ such that $\alpha\beta\gamma=0$. Then $S$ is a split by nilpotent extension of $R$ induced by the module generated by $\gamma$. Let $\lambda:R\to S$ be the corresponding ring homomorphism. Then $T=\mathbf{1}\oplus\mathbf{2}\oplus\begin{array}{c}\mathbf{3}\\ \mathbf{2}\\ \mathbf{1} \end{array}$ is a tilting $R$-module. If $0\to P^{-1}\overset{\sigma}\to P^{0}\to T\to 0$ is a minimal projective presentation for $T$, it is proved in \cite[Example 4.2]{As-Ma}, by using \cite[Lemma 2.2]{As-Ma}, that $T\otimes S$ is not of projective dimension at most $1$, so it is not an $1$-tilting $S$-module. However $\lambda^\star(T\otimes S)\in\Gen(T)$, hence $\sigma\otimes S$ is a $2$-term silting complex.
\end{example}

In the commutative case the silting property associated to $2$-term silting complexes ascends along all ring homomorphisms.

\begin{theorem}\label{commutative-1}
If $R$ and $S$ are commutative rings, $\lambda:R\to S$ is a unital ring homomorphism, and $\sigma:P^{-1} \to P^{0}$ is a $2$-term silting complex then the complex $\sigma\otimes_R S$ is silting. 
\end{theorem}

\begin{proof}
Let $T=\Coker(\sigma)$. In order to apply Theorem \ref{thm-1}, we will prove that $\lambda^\star(T\otimes_R S)$ is $T$-generated. Let $\alpha:R^{(I)}\to S\to 0$ be an epimorphism of $R$-modules. Since $R$ and $S$ are commutative, $\alpha$ is a homomorphism of $R$-$R$-bimodules, and it follows that $1_T\otimes_R \alpha:T\otimes_R R^{(I)}\to T\otimes_R S$ is an $R$-epimorphism.       
\end{proof}

\section{The descent property for commutative rings}

The main aim of this section is to prove that in the commutative case the property ``$2$-term silting complex'' descends along faithfully flat ring homomorphisms. We note that the restriction to faithfully flat ring homomorphisms is natural. 

\begin{example}
Let $\lambda:\BBZ\to \BBQ$ be the canonical embedding. Therefore, $\lambda$ is a ring epimorphism, but it is not faithfully flat. Let $0\to F^{-1}\overset{\sigma}\to F^0\to \BBQ\to 0$ be a projective presentation in $\Modr\BBZ$ for the group of rational numbers. Then $\CALD_\sigma=\Ker\Ext_\BBZ^1(\BBQ,-)$ contains all finite abelian groups, \cite[Property 52 (D)]{fuchs}. But for every finite group $G$, we have $\Hom_\BBZ(\BBQ,G)=0$, hence $\CALD_\sigma\neq \Gen(\BBQ)$. It follows that $\sigma$ is not a $2$-term silting complex. However, it is easy to see that   $\sigma\otimes_\BBZ\BBQ$ is a $2$-term silting complex of $\BBQ$-modules.  
\end{example}

In this section all rings are commutative. If $R$ is a commutative ring then $\Spec(R)$ will be the spectrum of $R$, and for every $\frakp\in\Spec(R)$ we will denote by $\kappa(\frakp)$ the field of fractions of $R/\frakp$. If $I$ is an ideal of $R$ then $V(I)=\{\frakp\in\Spec(R)\mid I\subseteq \frakp\}$. 
If $\lambda:R\to S$ is a faithfully flat ring homomorphism, then it is injective. Therefore, in order to simplify the presentation, we will often view $R$ as a subring of $S$. For instance, if $J$ is an subset of $S$ we will write $R\cap J$ instead of $R\cap \lambda^{-1}(J)$.   
We refer to \cite{Mats-comm} for other notations and for the basic properties which will be used here.

If $\sigma:P^{-1}\to P^0$ is a homomorphism between projective $R$-modules then we will associate to $\sigma$, as in \cite{An-Hr}, the class 
$$\CT_\sigma=\{M\in\Modr R\mid \sigma\otimes_R M \textrm{ is a monomorphism}\}.$$ Moreover, we also use the class
$$V_\sigma=\{\frakp\in\Spec(R)\mid \sigma\otimes_R\kappa(\frakp) \textrm{ is not a monomorphism}\}.$$ 

The class $\CT_\sigma$ is closed under submodules  and extensions (e.g. \cite[Lemma 2.2.2]{Br-Mo}). Moreover, using \cite[Lemma 3.3 and Lemma 4.2]{An-Hr}, we observe that if $\sigma$ is a $2$-term silting complex then $\CT_\sigma$ is the torsion-free class associated to a \textsl{hereditary} torsion theory of \textsl{finite type}  $(\CA_\sigma,\CT_\sigma)$, i.e. $\CT_\sigma$ is also closed under direct products, injective envelopes, and direct limits.

A class $\CALD$ of $R$-modules is called a \textsl{silting class} if there exists a $2$-term silting complex $\sigma$ such that $\CALD=\CALD_\sigma$ (i.e. there exits a silting module $T$ such that $\CALD=\Gen(T)$). 
In \cite[Theorem 4.7]{An-Hr} it is proved that there exists a bijective correspondence between the silting classes of a commutative ring and the Gabriel filters of finite type. We recall from \cite[Theorem 2.2]{Ga-Pr} that if $R$ is a commutative ring then there exist $1-1$ correspondences between the class of Gabriel filters of finite type defined on $R$, the class of hereditary torsion theories of finite type on $\Modr R$ and the set of \textsl{Thomason subsets} of $\Spec(R)$, i.e. unions of families of subsets of the the form $V(I)$ with $I$ finitely generated ideals (these are the open sets associated to Hochster's topology defined on $\Spec(R)$).
Therefore, we can enunciate the following

\begin{theorem}\cite[Theorem 4.7]{An-Hr}, \cite[Theorem 2.2]{Ga-Pr} \label{main-corr}
If $R$ is a commutative ring then there exist bijections between the following classes:
\begin{enumerate}[{\rm (1)}]
 \item the class of hereditary torsion theories of finite type in $\Modr R$;
 \item the class of Gabriel filters of finite type;
 \item the class of Thomason subsets in $\Spec(R)$;
 \item the class of silting classes in $\Modr R$. 
\end{enumerate}
\end{theorem}

These bijections are described in the above mentioned papers. For reader's convenience we list here the correspondences which will used in the following. We refer to \cite[Section 10.39]{staks-04VM} for the basic properties of the support. In particular, we recall that for every finitely presented $R$-module $M$ we have, by Nakayama's Lemma, $\supp(M)=\{\frakp\mid M\otimes_R\kappa(\frakp)\neq 0\}$ (see the proof of \cite[Lemma 10.39.8]{staks-04VM}). 

If $\CG$ is a Gabriel filter of finite type on $R$ and $\CB\subseteq \CG$ is a cofinal set of finitely generated ideals then 
\begin{itemize}
 \item 
the torsion free class associated to $\CG$ is $$\CF_\CG=\Ker\Hom_R(\bigoplus_{I\in \CB}R/I,-);$$  
\item the Thomason subset associated to $\CG$ is $$V_{\CG}=\{\frakp\in\Spec(R)\mid \exists I\in \CG \textrm{ such that } \frakp\in\supp R/I\}.$$ We can use \cite[Lemma 10.39.9]{staks-04VM} to conclude that $$V_\CG=\{\frakp\in\Spec(R)\mid \exists I\in \CB \textrm{ such that } R/I\otimes_R\kappa(\frakp)\neq 0\};$$ 
\item the silting class induced by $\CG$ is $$\CALD_\CG=\bigcap_{I\in \CG}\Ker(-\otimes_R R/I)=\bigcap_{I\in \CB}\Ker(-\otimes_R R/I).$$   
\end{itemize}

If $\sigma$ is a $2$-term silting complex then the Gabriel filter of finite type induced by $\CT_\sigma$ via \cite[Theorem 2.2]{Ga-Pr} is 
$$\CG_\sigma=\{I\leq R \mid \Hom_R(R/I,\CT_\sigma)=0\}.$$
In the bijective correspondence constructed in \cite[Theorem 4.7]{An-Hr} the Gabriel filter associated to each silting class $\CALD$  is defined as 
$$\CG_\CALD=\{I\leq R\mid M=IM \textrm{ for all }M\in\CALD\}.$$

\begin{lemma}\label{filtre}
If $\CALD$ is a silting class and $\sigma$ is a $2$-term silting complex such that $\CALD=\CALD_\sigma$ then $\CG_\CALD=\CG_\sigma$. 
\end{lemma}

\begin{proof}
Since in \cite[Lemma 3.3(2)]{An-Hr} it is proved that the definable classes $\CALD=\CALD_\sigma$ and $\CT_\sigma$ are dual, it follows from \cite[Corollary 3.4.21]{Prest} that a module $X$ belongs to $\CT_\sigma$ if and only if $X^+\in \CALD$, where $X^+=\Hom_\BBZ(X,\BBQ/\BBZ)$. Therefore, the correspondence $\CALD\mapsto \CG_\sigma$ is independent of the choice of $\sigma$. Moreover, a module $M$ is in $\CALD$ if and only if $X^+\in \CT_\sigma$. By \cite[Theorem 2.2]{Ga-Pr} this is equivalent to $\Hom(R/I,X^+)=0$ for all $I\in \CG_{\sigma}$. Using the natural isomorphism induced by $\Hom$ and $\otimes$, we obtain $M\in \CALD$ if and only if $R/I\otimes_R M=0$ for all $I\in \CG_{\sigma}$. This means that $\CG_{\sigma}$ is the Gabriel filter of finite type constructed in \cite[Proposition 4.4]{An-Hr}, and the proof is complete. 
\end{proof}

In the following we present some properties of the classes involved in Theorem \ref{main-corr}. We start with a well-known lemma.

\begin{lemma}\label{nat-kp}
If $\frakp\in \Spec(R)$ and $E(R/\frakp)$ is the injective envelope for $R/\frakp$ then 
\begin{enumerate}[{\rm (i)}]
\item $E(R/\frakp)$ is an injective cogenerator for $\Modr R_\frakp$;
 \item there is a natural isomorphism $\Hom_R(-,\kappa(\frakp))\cong \Hom_{R_\frakp}(-\otimes_R\kappa(\frakp),E(R/\frakp)).$ 
\end{enumerate}
\end{lemma}

\begin{proof}
The first statement is well known. For the the second statement, we observe that there are the natural isomorphisms  
$$\Hom_R(-,\kappa(\frakp))\cong\Hom_R(-,\Hom_{R_\frakp}(\kappa(\frakp),E(R/\frakp)))\cong \Hom_{R_\frakp}(-\otimes_R\kappa(\frakp),E(R/\frakp)),$$ so we have the required isomorphism. 
\end{proof}

\begin{corollary}\label{p-nin-V}
Let $\CG$ be a Gabriel filter of finite type, and $\frakp\in\Spec(R)$. Then the following are equivalent: \begin{enumerate}[{\rm (i)}] 
\item $\frakp\notin V_\CG$;
\item $\kappa(\frakp)\in\CF_\CG$;
\item $R/\frakp\in\CF_\CG$.
\end{enumerate}

In particular, if $\sigma$ is a $2$-term silting complex, and $\CG$ is the induced Gabriel filter, then $\CF_{\CG}=\CT_\sigma$, and $V_\CG=V_\sigma$.
\end{corollary}

\begin{proof}
The equivalence (i)$\Leftrightarrow$(ii) follows from Lemma \ref{nat-kp}, while (ii)$\Leftrightarrow$(iii) is true since $\CF_\CG$ is closed under submodules and injective envelopes.  

If $\sigma$ is a $2$-term silting complex, the equality $\CF_{\CG}=\CT_\sigma$ is true since $\CG$ is the Gabriel filter associated to $\CT_\sigma$. The second equality follows from the equivalence (i)$\Leftrightarrow$(ii).
\end{proof}

\begin{lemma}\label{Csigma} Let $\lambda:R\to S$ be a homomorphism of commutative rings. If $P^{-1}$ and $P^0$ are projective $R$-modules and $\sigma:P^{-1}\to P^0$ is a homomorphism, we denote by $\CT_{\sigma\otimes_R S}\subseteq \Modr S$ the class associated to the $S$-homomorphism $\sigma\otimes_R S$. The following are true:
\begin{enumerate}[{\rm (i)}]
 \item $\CT_{\sigma\otimes_R S}=\CT_\sigma\cap\Modr S$.
\item Suppose that $\lambda$ is faithfully flat. For a module $M\in \Modr R$ we have $M\in\CT_\sigma$ if and only if $M\otimes_R S\in\CT_{\sigma\otimes_R S}$.  
 \end{enumerate}
\end{lemma}

\begin{proof}
(i) This follows by using the natural isomorphisms $\sigma\otimes_R S\otimes_S M\cong \sigma\otimes_R M$ for all $M\in \Modr S$.

(ii) Suppose that $M\in \CT_\sigma$. Then $\sigma\otimes_R M$ is monic. Since $S$ is flat, it follows that $\sigma\otimes_R M\otimes_R S$ is monic, hence $\lambda^\star(M\otimes_R S)\in \CT_\sigma$. By (i) we obtain $M\otimes_R S\in \CT_{\sigma\otimes_R S}$. 

Conversely, suppose that $M\otimes_R S\in \CT_{\sigma\otimes_R S}$. Then $\lambda^\star(M\otimes _R S)\in \CT_\sigma$. Since $S$ is faithfully flat we know that $M$ can be embedded as a submodule of $\lambda^\star(M\otimes S)$, hence $M\in\CT_\sigma$. 
\end{proof}

In the following, we will use, as in \cite{Br_Ze:2014} the notation $\Def_\sigma=\Coker(\Hom_R(\sigma,-)):\Modr R\to Ab$. Therefore, for every homomorphism $\sigma$ we have $\CALD_\sigma=\{M\in\Modr R\mid \Def_\sigma(M)=0\}$. 

\begin{lemma}\label{transfer}
Let $\lambda:R\to S$ be a homomorphism of commutative rings. If $L^{-1}$ and $L^0$ are finitely generated and projective $R$-modules and $\sigma:L^{-1}\to L^0$ is an $R$-morphism then we have a natural isomorphism 
 $$\Def_{\sigma\otimes_R S}(-\otimes_R S)\cong \Def_{\sigma}(-)\otimes_R S.$$
\end{lemma}

\begin{proof}
This is a consequence of the fact that for every finitely presented $R$-module $L$, there exists a natural isomorphism between the functors  $\Hom_S(L\otimes_RS,-\otimes_RS)$ and $\Hom_R(L,-)\otimes_RS$ (see \cite[Theorem 7.11]{Mats-comm}).  
\end{proof}

\begin{lemma}\label{var2.3}
Suppose that $\lambda:R\to S$ is a homomorphism of commutative rings and that $\sigma$ is a $2$-term silting complex. Then
\begin{enumerate}[{\rm (a)}]
 \item $\CALD_\sigma\otimes_R S\subseteq \CALD_{\sigma\otimes_R S}$;
 \item if $\lambda$ is faithfully flat and $N\in\Modr R$ then $N\in\CALD_\sigma$ if and only if $N\otimes_R S\in \CALD_{\sigma\otimes_R S}$.
\end{enumerate}
\end{lemma}

\begin{proof}
(a) If $N\in \CALD_\sigma$ then $N\otimes S$ is $T$-generated. Since it is an $S$-module we obtain $N\otimes_R S\in \CALD_{\sigma\otimes_R S}$ by Theorem \ref{thm-1}.

(b) Suppose that $N\otimes_R S\in \CALD_{\sigma\otimes_R S}$. It was proved in \cite[Theorem 2.3]{An-Hr} and in \cite[Theorem 6.3]{Ma-Sto:2015} that every silting class is of finite type. Therefore, there exists a family $\sigma_i$, $i\in I$, of homomorphisms between finitely generated projective $R$-modules such that $\CALD_\sigma=\bigcap_{i\in I}\CALD_{\sigma_i}$. 
By Lemma \ref{basic1} it follows that $N\otimes_R S\in \CALD_\sigma$.  
Using the proof of Lemma \ref{basic1} and Lemma \ref{transfer}, we observe that for every $i\in I$ we have the isomorphisms $$0=\Def_{\sigma_i}(N\otimes S)\cong \Def_{\sigma_i\otimes_R S}(N\otimes_R S)\cong \Def_{\sigma_i}(N)\otimes_R S.$$ Since $S$ is faithfully flat, it follows that $N\in\CALD_{\sigma_i}$ for all $i\in I$, so $N\in\CALD_\sigma$.
\end{proof}

\begin{remark}
We recall from \cite{AMV:2015} that if $\sigma$ is $2$-term silting complex then it is a \textsl{generator} in $\BD(R)$, i.e. the smallest triangulated subcategory which contains $\sigma$ and is closed under direct sums is $\BD(R)$. 

Moreover, it was proved in \cite[Theorem 3.14]{An-18} (see also \cite[Remark 2.7]{AMV:2015}) that an object in $\BD(R)$ is a generator if and only if it has the property: for every $Y\in \BD(R)$, from $\Hom(X,Y[i])=0$ for all $i\in \mathbb{Z}$ it follows that $Y=0$. 
\end{remark}

For the proof of the following lemma we use the same techniques as those used in the proof of \cite[Lemma 4.12]{Hr-St-Tr}. 

\begin{lemma}\label{generator}
Let $\sigma:P^{-1}\to P^{0}$ be a homomorphism between projective $R$-modules with $\Coker(\sigma)=T$. If the complex (concentrated in $-1$ and $0$) which is induced by $\sigma$ is a generator for $\mathbf{D}(R)$ then for every $\mathfrak{p}\in\Spec(R)$ we have $T\otimes_R\kappa(\frakp)\neq 0$ or $\Ker(\sigma\otimes_R\kappa(\frakp))\neq 0$.
\end{lemma}

\begin{proof}
Since $\sigma$ is a generator for $\mathbf{D}(R)$, for every $R$-module $M$ there exists $i\in\{0,1\}$ such that $\Hom_{\mathbf{D(R)}}(\sigma,M[i])=\Hom_{\mathbf{K(R)}}(\sigma,M[i])\neq 0$. It follows that $\Hom(\sigma,M)$ is not an isomorphism for all $R$-modules $M$.

If $\frakp\in \Spec(R)$, we take $M=\kappa(\frakp)=\Hom_{R_\frakp}(\kappa(\frakp),E(R/\frakp))$, where $E(R/\frakp)$ is the injective envelope of $\kappa(\frakp)$ in $\Modr R_\frakp$. Using Lemma \ref{nat-kp} we obtain that $\Hom_R(\sigma,\kappa(\frakp))$ is not an isomorphism if and only if $\sigma\otimes_R\kappa(\frakp)$ is not an isomorphism.
\end{proof}

We recall from \cite{silver} and \cite{staks-04VM} some basic facts about ring epimorphisms.

\begin{lemma}\label{epimorfisms}
Let $R$ be a commutative ring and $\delta:R\to B$ a ring epimorphism. Then
\begin{enumerate}[{\rm (i)}]
 \item $B$ is commutative;
 \item the canonical map $\delta^\star:\Spec(B)\to \Spec(R)$ is injective;
 \item for every $\frakq\in \Spec(B)$, $\kappa(\frakq)=\kappa(\delta^\star(\frakq))$.
\end{enumerate}
\end{lemma}

\begin{proposition}\label{ps-s-k(p)}
Let $\sigma:P^{-1}\to P^{0}$ be a homomorphism between projective $R$-modules. The following are equivalent:
\begin{enumerate}[{\rm (1)}]
 \item $\sigma$ is a $2$-term silting complex;
 \item \begin{enumerate}[{\rm (i)}] 
 \item $\sigma$ is partial silting,
 \item for every $\mathfrak{p}\in\Spec(R)$ we have $T\otimes_R\kappa(\frakp)\neq 0$ or $\Ker(\sigma\otimes_R\kappa(\frakp))\neq 0$.
 \end{enumerate}
\end{enumerate}
\end{proposition}

\begin{proof} (1)$\Rightarrow$(2) This follows from Lemma \ref{generator}.

(2)$\Rightarrow$(1) Since $\sigma$ is partial silting, then $\Gen(T)\subseteq \CALD_\sigma$ are torsion classes, and the torsion-free class corresponding to $\Gen(T)$ is $\Ker\Hom_R(T,-)$, \cite[Remark 3.8]{AMV:2015}. Therefore, in order to obtain $\Gen(T)=\CALD_\sigma$, it is enough to prove that $\CALD_\sigma\cap \Ker\Hom_R(T,-)=0$. 

Let $\CY=\CALD_\sigma\cap \Ker\Hom_R(T,-)$. By \cite[Proposition 3.3]{AMV2} the class $\CY$ is bireflective and extension closed. Therefore, it induces an epimorphism of rings $\delta:R\to B$, and $\CY$ is the essential image of the restriction of scalars $\delta^\star:\Modr B\to\Modr R$.    

Suppose that $B\neq 0$. Then $\Spec(B)\neq \varnothing$. By Lemma \ref{epimorfisms}  it follows that for every $\frakq\in\Spec(B)$ we can identify $\kappa(\frakq)=\kappa(\frakp)$ for some $\frakp\in\Spec(R)$. Therefore, there exists $\frakp\in\Spec(R)$ such that $\kappa(\frakp)\in\CY$. If $\frakp$ is such an ideal we obtain that $\Hom_R(\sigma,\kappa(\frakp))$ is an isomorphism. Using Lemma \ref{nat-kp}, it follows that $\sigma\otimes_R \kappa(\frakp)$ is an isomorphism, and this contradicts (ii). Then $B=0$, and the proof is complete.  
\end{proof}

The following properties are well known. 

\begin{lemma}\label{p-si-q} Let $\lambda:R\to S$ be a faithfully flat homomorphism of commutative rings. If $\frakq\in\Spec(S)$ and $\frakp=\frakq\cap R$ then there is a natural isomorphism  
$$-\otimes_R S\otimes_S\kappa(\frakq)\cong -\otimes_R \kappa(\frakp)\otimes_{\kappa(\frakp)}\kappa(\frakq).$$ Moreover $\kappa(\frakq)$ is faithfully flat as a $\kappa(\frakp)$-module. 
\end{lemma}

We recall that we use the notation $$V_\sigma=\{\frakp\in\Spec R\mid \sigma\otimes_R \kappa(\frakp) \textrm{ is not a monomorphism}\}$$ for every homomorphism $\sigma$ between projective modules.

\begin{lemma}\label{lemmaA}
Suppose that $\lambda:R\to S$ is faithfully flat, and $\sigma:P^{-1}\to P^0$ is a homomorphism between projective $R$-modules such that $\sigma\otimes_R S$ is a $2$-term silting complex. Let $\lambda^\star:\Spec(S)\to \Spec(R)$, $\lambda^\star(\frakq)=\frakq\cap R$, be the canonical map. Then
\begin{enumerate}[{\rm (i)}]
\item for an ideal $\frakq\in\Spec(S)$, we have $\kappa(\frakq)\in \CT_{\sigma\otimes S}$ if and only if $\kappa(\lambda^\star(\frakq))\in \CT_\sigma$;   

\item $\lambda^\star (\Spec(S)\setminus V_{\sigma\otimes_R S})=\Spec(R)\setminus V_\sigma$; 

\item $\lambda^\star (V_{\sigma\otimes_R S})=V_\sigma.$
\end{enumerate}
\end{lemma}

\begin{proof}
(i) If $\frakq\in \Spec(S)$ and $\frakp=\frakq\cap R$, it follows from Lemma \ref{p-si-q} that $\sigma\otimes_R \kappa(\frakp)$ is injective if and only if $\sigma\otimes_R \kappa(\frakp)\otimes_{\kappa(\frakp)}\kappa(\frakq)$ is injective, and this is equivalent to $\sigma\otimes_R S\otimes_S \kappa(\frakq)$ is injective.

(ii) We observe that 
$$V_{\sigma}=\{\frakp\in\Spec(S)\mid \kappa(\frakp)\notin \CT_{\sigma}\}\textrm{ and }
V_{\sigma\otimes_R S}=\{\frakq\in\Spec(S)\mid \kappa(\frakq)\notin \CT_{\sigma\otimes_R S}\}.$$
By (i) it follows that for an ideal $\frakq\in\Spec(S)$ we have $\frakq\notin V_{\sigma\otimes_R S}$ if and only if $\lambda^\star(\frakq)\notin V_\sigma$. 

Since $\lambda^\star$ is surjective, it follows that $\lambda^\star (\Spec(S)\setminus V_{\sigma\otimes_R S})=\Spec(R)\setminus V_\sigma$.

(iii) This follows by using (ii) and the surjectivity of $\lambda^*$.
\end{proof}

In the end of these preliminary considerations we recall the results obtained in \cite{Hr-St-Tr} for the study of the descent of $1$-tilting modules. 

\begin{proposition}\cite[Section 4]{Hr-St-Tr} \label{tilting-desc}
Let $\overline{\lambda}:\overline{R}\to \overline{S}$ be a faithfully flat homomorphism of rings. If $T$ and $V$ are $\overline{R}$-modules such that 
\begin{enumerate}[{\rm (a)}]
 \item $V$ is tilting, 
 \item $T\otimes_{\overline{R}}\overline{S}$ is a tilting $\overline{S}$-module, and
 \item $\Gen(T\otimes_{\overline{R}} \overline{S})=\Gen(V\otimes_{\overline{R}} \overline{S})$
\end{enumerate}
then $T$ is a tilting $\overline{R}$-module and $\Gen(T)=\Gen(V)$. 
\end{proposition}

We are ready to prove the descent property for $2$-term silting complexes.

\begin{theorem}\label{silting}
Suppose that $\lambda:R\to S$ is a faithfully flat ring homomorphism. If $\sigma:P^{-1}\to P^0$ is a homomorphism in $\Modr R$ such that $\sigma\otimes_R S$ is a $2$-term silting complex of $S$-modules then $\sigma$ is a $2$-term silting complex of $R$-modules. 
\end{theorem}

\begin{proof} Since $\sigma\otimes_R S$ is a $2$-term silting complex it follows that $P^{-1}\otimes_RS$ and $P^{0}\otimes_RS$ are projective $S$-modules. Using the descend property of projective modules, \cite{Gr-Ra}, it follows that the $R$-modules $P^{-1}$ and $P^{0}$ are projective. 

Let $\lambda^\star:\Spec(S)\to \Spec(R)$, $\lambda^\star(\frakq)=\frakq\cap R$ be the canonical map.
If $V_{\sigma\otimes_R S}\subseteq \Spec(S)$ is the Thomason set associated to $\sigma\otimes S$ then we use \cite[Lemma 3.15]{Hr-St-Tr} together with Lemma \ref{lemmaA}(ii) to conclude that $\lambda^\star(V_{\sigma\otimes_R S})$ is a Thomason subset in $\Spec(R)$. By Theorem \ref{main-corr} and Corollary \ref{p-nin-V} there exists a $2$-term silting complex $\rho:L^{-1}\to L^{0}$ in $\Modr R$ such that $V_\rho=\lambda^\star(V_{\sigma\otimes_R S})$. Using the ascent property proved in Theorem \ref{commutative-1} we conclude that $\rho\otimes_R S$ is silting in $\Modr S$.  

We apply Lemma \ref{lemmaA}(iii) to $\rho$ and $\sigma$, and we obtain that $$\lambda^\star(V_{\rho\otimes_R S})=V_\rho=\lambda^\star(V_{\sigma\otimes_R S})=V_\sigma.$$
We use Lemma \ref{lemmaA}(i) and we obtain the equality $V_{\rho\otimes_R S}=V_{\sigma\otimes_R S}$. It follows that the homomorphisms $\rho\otimes_R S$ and $\sigma\otimes_R S$ induce the same silting class, i.e. $\CALD_{\sigma\otimes_R S}=\CALD_{\rho\otimes_R S}$.

\medskip

Let us denote $T=\Coker(\sigma)$ and
$\CG$ be the Gabriel filter associated to $\rho$ as in Theorem \ref{main-corr} and Lemma \ref{filtre}. It follows that 
$$\CALD_\rho=\bigcap_{I\in\CG}\Ker(-\otimes_R R/I).$$ 
Since $T\otimes S\in \CALD_{\rho\otimes_R S}$, we can use Lemma \ref{var2.3} to conclude that $T\in \CALD_\rho$. Therefore, for every $I\in \CG$ we have $T\otimes_R R/I=0$. 

\begin{claim}\label{claim-1}
For every $I\in\CG$ there exists a set $K$ and a pushout diagram
$$\xymatrix{(P^{-1})^{(K)}\ar[r]^{\sigma^{(K)}}\ar[d]^{\delta} & (P^{0})^{(K)}\ar[r]\ar[d] & T^{(K)} \ar[r]\ar@{=}[d] & 0\\
R\ar[r]^\alpha & E\ar[r] & T^{(K)}\ar[r] &0,
}$$
such that $E\otimes_R R/I=0$. 
\end{claim}

Let $I\in \CG$ be a fixed ideal. We will use the notation $\sigma\otimes_R R/I=\widehat{\sigma}_{R/I}$. 

Applying $-\otimes_R R/I$ to $\sigma$ we obtain a short exact sequence of $R/I$-modules 
$$0\to \Ker(\widehat{\sigma}_{R/I})\overset{u}\to P^{-1}\otimes_R R/I\overset{\widehat{\sigma}_{R/I}}\longrightarrow P^0\otimes_R R/I\to 0,$$ 
which splits since $P^0\otimes_R R/I$ is projective. Therefore $\Ker(\widehat{\sigma}_{R/I})$ is projective. 

Moreover, since this exact sequence splits, for every $\frakp\in\Spec(R/I)$ we obtain a commutative diagram 
$$\xymatrix{\Ker(\widehat{\sigma}_{R/I})\otimes_{R/I}\kappa(\frakp) \ar@{>->}[r]^{u\otimes_{R/I}\kappa(\frakp)} \ar[d] & P^{-1}\otimes_{R} R/I\otimes_{R/I}\kappa(\frakp)\ar@{->>}[r]\ar[d] & P^0\otimes_{R} R/I\otimes_{R/I}\kappa(\frakp) \ar[d] \\
 \Ker(\sigma\otimes_{R}\kappa(\frakp))\ar@{>->}[r] & P^{-1}\otimes_{R}\kappa(\frakp)\ar@{->>}[r] & P^0\otimes_{R}\kappa(\frakp) 
}$$
such that the horizontal lines are (split) short exact sequences and the vertical arrows are isomorphisms. 
Since  $V_\sigma$ is the Thomason set corresponding to $\CG$, we have  $V(I) \subseteq V_\sigma$. It follows that $\Ker(\widehat{\sigma}_{R/I})\otimes_{R/I}\kappa(\frakp)\neq 0$ for all $\frakp\in \Spec(R/I)$, hence  $\Ker(\widehat{\sigma}_{R/I})$ is a projective generator for $\Modr R/I$. 

Therefore, there exists a set $K$ and an $R/I$-epimorphism $\gamma:\Ker(\widehat{\sigma}_{R/I})^{(K)}\to R/I$. We also fix a homomorphism $v:(P^{-1})^{(K)}\otimes_R R/I\to \Ker(\widehat{\sigma}_{R/I})^{(K)}$ such that $vu^{(K)}=1$. If $\pi:R\to R/I$ is the canonical surjective homomorphism, then 
there exists a homomorphism $\delta:(P^{-1})^{(K)}\otimes_R R\to R$ such that $\pi\delta=\gamma v (1\otimes_R\pi)$. 
In order to simplify the presentation we identify $(P^{-1})^{(K)}\otimes_R R$ with $(P^{-1})^{(K)}$, and all these data are represented in the following commutative diagram:
$$\xymatrix{ &(P^{-1})^{(K)} \ar[rr]^{\sigma^{(K)}}\ar@{-->}[rdd] \ar[d]_{1\otimes_R \pi} && (P^0)^{(K)}\ar[d]^{1\otimes_R \pi}\\
\Ker(\widehat{\sigma}_{R/I})^{(K)}\ar@{>->}[r]^{u^{(K)}}\ar@{->>}[d]^\gamma & (P^{-1}\otimes_{R} R/I)^{(K)}\ar@{->>}[rr]^{(\sigma\otimes_{R}R/I)^{(K)}} \ar@/_2pc/[l]_v && (P^{0}\otimes_{R} R/I)^{(K)}\\
R/I\ar@{<<-}[rr]^\pi&&R&\ ,
}$$
where the dashed arrow is $\delta$. We apply the functor $-\otimes_R R/I$ to this diagram, and we obtain the commutative diagram
$$\xymatrix{\Ker(\sigma^{(K)} \otimes_{R} R/I)\ar@{>->}[rr]^{u^{(K)}}\ar@{-->}[d]^{\alpha} &&(P^{-1})^{(K)}\otimes_{R} R/I \ar@/^2pc/[rdd]^{\delta\otimes_R R/I} \ar[d]_{1\otimes_R \pi\otimes_R R/I} & 
\\
\Ker(\widehat{\sigma}_{R/I})^{(K)}\otimes_{R} R/I\ar@{>->}[rr]^{u^{(K)}\otimes_R R/I}\ar@{->>}[d]^{\gamma\otimes_R R/I} && (P^{-1}\otimes_{R} R/I)^{(K)}\otimes_{R} R/I \ar@/_2pc/[ll]_{v\otimes_R R/I} & \\
R/I\otimes_{R} R/I\ar@{<<-}[rrr]^{\pi\otimes_R R/I}&&&R\otimes_{R} R/I,
}$$ where $\alpha$ is the canonical map. Using the obvious identifications and the natural isomorphisms $R/I\otimes_R R/I\cong R/I\otimes_{R/I} R/I\cong R/I$, \cite[Proposition II.2]{Bo}, it follows that $\alpha$ and $\pi\otimes_R R/I$ are isomorphisms. It is not hard to conclude that $(\delta\otimes_R R/I)u^{(K)}$ is an epimorphism.

We construct the pushout diagram 
$$\xymatrix{(P^{-1})^{(K)}\ar[r]^{\sigma^{(K)}}\ar[d]^{\delta} & (P^{0})^{(K)}\ar[r]\ar[d] & T^{(K)} \ar[r]\ar@{=}[d] & 0\\
R\ar[r]^\alpha & E\ar[r] & T^{(K)}\ar[r] &0.
}$$ Applying the tensor product $-\otimes_R R/I$, and using the commuting property of the tensor product with respect to direct sums together with the well-known fact that the direct sums preserve exact sequences we obtain the commutative diagram 
$$\xymatrix{ \Ker(\widehat{\sigma}_{R/I})^{(K)}\ar@{>->}[r]^{u^{(K)}}\ar[d] & (P^{-1}\otimes_{R} R/I)^{(K)}\ar@{->>}[rr]^{(\widehat{\sigma}_{R/I})^{(K)}}\ar[d]^{\delta\otimes_{R} R/I}  && (P^{0}\otimes_{R} R/I)^{(K)}\ar[d]\\
 \Ker(\alpha\otimes_{R} R/I)\ar@{>->}[r] & R\otimes_{R} R/I \ar@{->>}[rr]^{\alpha\otimes_R R/I} && E\otimes_{R} R/I.
}$$
Since $(\delta\otimes_R R/I)u^{(K)}$ is an epimorphism, it follows that  $E\otimes_R R/I=0$. 

\begin{claim}\label{claim-2}
$\CALD_\sigma\subseteq \CALD_\rho$. 
\end{claim}

In order to prove this, let us fix a module $M\in \CALD_\sigma$. For every $I\in\CG$ we construct a pushout diagram as in Claim \ref{claim-1}.

We consider an epimorphism $f:R^{(L)}\to M$. Then $f\delta^{(L)}:[(P^{-1})^{(K)}]^{(L)}\to M$ can be extended to a homomorphism $[(P^{0})^{(K)}]^{(L)}\to M$ through $[\sigma^{(K)}]^{(L)}$. It follows that there exists a homomorphism $g:E^{(L)}\to M$ $f$ such that $f=g\alpha{(L)}$. Since $f$ is an epimorphism, we also have that $g$ is an epimorphism. Using the property $E\otimes_R R/I=0$, we obtain $M\otimes_R R/I=0$.

It follows that $M\otimes_R R/I=0$ for all $I\in \CG$. Therefore, $M\in \CALD_\rho$, and the proof for the inclusion $\CALD_\sigma\subseteq \CALD_\rho$ is complete.

\begin{claim}\label{claim-3}
If $V=\Coker(\rho)$ then $\Ann(T)=\Ann(V)$. 
\end{claim}

Let $U=\Ann(V)$. We have $US\subseteq \Ann(V\otimes_R S)=\Ann(T\otimes_R S)$. But $S$ is faithfully flat, hence $U\subseteq US$. We can view $T$ as a submodule of $T\otimes_R S$. Therefore, $UT=0$, and it follows that $U\subseteq \Ann(T)$. 

Since in this proof we only used the equality $\Ann(V\otimes_R S)=\Ann(T\otimes_R S)$, it follows that the converse inclusion is valid, so the proof for Claim \ref{claim-3} is complete. 

\medskip

In the following we will use the notations $U=\Ann(V)=\Ann(T)$, $\overline{R}=R/U$, and $\overline{S}=S/US$.

\begin{claim}\label{claim-4}
The $\overline{R}$-modules $V$ and $T$ are tilting and $\Gen(T)=\Gen(V)$.  
\end{claim}

We view $\Modr \overline{R}$ as a full subcategory of $\Modr R$ via the canonical homomorphism $R\to R/U$. Then for every $M\in \Modr \overline{R}$ we have $UM=0$, and it follows that 
the left side homomorphism in the exact sequence $$M\otimes_R US\to M\otimes_R S\to M\otimes_R \overline{S}\to 0$$ is zero. Therefore, the restrictions of the functors $-\otimes_R S$ and $-\otimes_R \overline{S}$ to $\Modr \overline{R}$ are naturally isomorphic. Since $\overline{S}$ is also an $\overline{R}$-module, we also have a natural isomorphism  $-\otimes_R \overline{S}\cong -\otimes_{\overline{R}} \overline{S}$ for the restrictions of these functors to $\Modr \overline{R}$, \cite[Proposition II.2]{Bo}. This shows that the canonical homomorphism $\overline{\lambda}:\overline{R}\to \overline{S}$, induced by $\lambda$, is faithfully flat.   

From \cite[Proposition 3.2 and Proposition 3.10]{AMV:2015} it follows that $V$ is a tilting $\overline{R}$-module. By \cite[Lemma 2.4]{Hr-St-Tr} we have that $V\otimes_R S=V\otimes_{\overline{R}}\overline{S}$ is a tilting $\overline{S}$-module. This implies that the annihilator of the $\overline{S}$-module $V\otimes_{\overline{R}}\overline{S}$ is zero, and it follows that $\Ann_{\overline{S}}(T\otimes_{\overline{R}}\overline{S})=0$. Since $\sigma\otimes_R \overline{S}$ is a $2$-term silting complex, we can use \cite[Proposition 3.2 and Proposition 3.10]{AMV:2015} one more time to obtain that $T\otimes_R \overline{S}=T\otimes_{\overline{R}}\overline{S}$ is a tilting $\overline{S}$-module. By Proposition \ref{tilting-desc}, 
we obtain that $T$ is tilting as an $\overline{R}$-module and $\Gen(T)=\Gen(V)$. 

\begin{claim}\label{claim-5}
We have $\CALD_\rho\subseteq \CALD_\sigma$. In particular $T^{(I)}\in \CALD_\sigma$ for all sets $I$. 
\end{claim}

In $\Modr \overline{R}$ we have the projective resolution 
$$P^{-1}\otimes_R \overline{R}\overset{\sigma\otimes_R\overline{R}}\longrightarrow P^{0}\otimes_R \overline{R}\longrightarrow T\to 0.$$ But $T$ is tilting as an $\overline{R}$-module, hence it is of projective dimension at most $1$. It follows that we can find in $\Modr \overline{R}$ a direct decomposition $P^{-1}\otimes_R \overline{R}=\overline{P}\oplus \overline{K}$, where $\overline{K}=\Ima(\sigma\otimes_R\overline{R})$ and $\overline{P}=\Ker(\sigma\otimes_R\overline{R})$. 

Since the complex $\sigma\otimes_R S$ is silting in $\Modr S$, we can apply Theorem \ref{commutative-1} to conclude that 
$$\sigma\otimes_R \overline{S}\cong \sigma\otimes_R\overline{R}\otimes_{\overline{R}}\overline{S}\cong \sigma\otimes_R S\otimes_{\overline{S}}\overline{S} $$ is $2$-term silting complex in $\Modr \overline{S}$. Moreover, since the induced ring homomorphism $\overline{\lambda}:\overline{R}\to \overline{S}$ is faithfully flat, and $T$ is a tilting $\overline{R}$-module, we obtain by using \cite[Lemma 2.4]{Hr-St-Tr} that $T\otimes_{\overline{R}}\overline{S}$ is tilting.
It follows that $$\CALD_{\sigma\otimes_{\overline{R}}\overline{S}}=\Gen(T\otimes_{\overline{R}}\overline{S})= 
\Ker\Ext^1_{\overline{S}}(T\otimes_{\overline{R}}\overline{S},-),$$ hence $$\Hom_{\overline{S}}(\overline{P}\otimes_{\overline{R}}\overline{S},\Gen(T \otimes_{\overline{R}}\overline{S}))=0.$$ Since $\overline{S}$ is faithfully flat, the last equality implies that $\Hom_{\overline{R}}(\overline{P},\Gen(T))=0$.

Let $X\in\CALD_\rho=\Gen(V)=\Gen(T)$. Then $U\leq\Ann(X)$. If $f:P^{-1}\to X$ is a homomorphism, 
we obtain the commutative diagram 
$$ \xymatrix{0\ar[r] &P^{-1}\otimes_R U\ar[r]\ar[d]&P^{-1}\otimes_R R\ar[r]\ar[d]^{f\otimes_R R} & P^{-1}\otimes_R \overline{R}\ar[d]\ar[r]\ar@{-->}[ldd]_<<<<<<<<{\overline{f}} & 0 \\
&X\otimes_R U\ar[r]\ar[d] &X\otimes_R R\ar[r]\ar[d]& X\otimes_R \overline{R}\ar[r]\ar[d] & 0\\
0\ar[r] & XU\ar[r] &X \ar[r]& X/XU\ar[r] & 0,
}$$ where the vertical arrows in the bottom rectangle are the natural ones. Since $XU=0$, it follows that  
$f$ factorizes through $P^{-1}\otimes_R \overline{R}$. Therefore, there exists $\overline{f}:P^{-1}\otimes_R \overline{R}\to X$ such that $f=\overline{f}h$, where $h:P^{-1}\overset{\cong}\to P^{-1}\otimes_R R\to P^{-1}\otimes_R \overline{R}$ is the canonical map.

But $T$ is tilting as an $\overline{R}$-module. It follows that there exists a homomorphism $g:P^0\otimes_R \overline{R}\to X$ such that $\overline{f}_{|\overline{K}}=g(\sigma\otimes_R\overline{R})_{|\overline{K}}$. Since $\Hom(\overline{P},X)=0$, we have $\overline{f}(\overline{P})=0$. Then $\overline{f}=g(\sigma\otimes_R\overline{R})$, and we obtain the commutative diagram 
$$\xymatrix{ P^{-1} \ar[rr]^{\sigma} \ar[d]_h && P^0 \ar[r]\ar[d] & T \ar[r] \ar@{=}[d] & 0 \\
P^{-1}\otimes_R \overline{R} \ar[rr]^{\sigma\otimes_R\overline{R}} \ar[d]_{\overline{f}} && P^0\otimes_R \overline{R} \ar[r]\ar[lld]^g & T \ar[r]  & 0 ,\\
X & & &
}$$ 
where the composition of the vertical left side arrows is $f$. Then $f$ factorizes through $\sigma$, hence $X\in\CALD_\sigma$. 

We conclude that $\CALD_\rho\subseteq \CALD_\sigma$. 

\medskip

Using Claim \ref{claim-2} and Claim \ref{claim-5} we obtain $\CALD_\sigma= \CALD_\rho$, and $T\in \CALD_\sigma$. It follows that $\sigma$ is partial silting. By Proposition \ref{ps-s-k(p)}, in order to complete the proof, it is enough to prove 

\begin{claim}\label{claim-6}
For every $\mathfrak{p}\in\Spec(R)$ we have $T\otimes_R\kappa(\frakp)\neq 0$ or $\Ker(\sigma\otimes_R\kappa(\frakp))\neq 0$.
\end{claim}

Let $\frakp\in\Spec(R)$. Since $\lambda$ is faithfully flat, the induced map of spectra $$\lambda^\star:\Spec(S)\to \Spec(R),\ \lambda^\star(\frakq)=\frakq\cap R,$$ is surjective, so there exists $\frakq\in\Spec(S)$ such that $\frakq\cap R=\frakp$.  By Lemma \ref{p-si-q} we observe that $\kappa(\frakq)$ is faithfully flat as a $\kappa(\frakp)$-module (see also the proof of \cite[Proposition 3.16]{Hr-St-Tr}). Applying the functors from Lemma \ref{p-si-q} to $\sigma$ it follows that $$\Ker(\sigma\otimes_R S\otimes_S\kappa(\frakq))\cong \Ker(\sigma\otimes_R \kappa(\frakp))\otimes_{\kappa(\frakp)}\kappa(\frakq),$$ and 
$$T\otimes_R S\otimes_S\kappa(\frakq)\cong T\otimes_R \kappa(\frakp)\otimes_{\kappa(\frakp)}\kappa(\frakq).$$ 
By Proposition \ref{ps-s-k(p)}, for every $\frakq\in \Spec(S)$ we have $T\otimes_R S\otimes_S\kappa(\frakq)\neq 0$ or $\Ker(\sigma\otimes_R S\otimes_S\kappa(\frakq))\neq 0$. Since $\kappa(\frakq)$ is faithfully flat as a $\kappa(\frakp)$-module, we obtain Claim \ref{claim-6}, and the proof is complete.  
\end{proof}

We close the paper with some comments on the proof of Theorem \ref{silting}.

\begin{remark} (a) From Claim \ref{claim-3} and Claim \ref{claim-4} it follows that $T$ is tilting as an $R/\Ann(T)$-module. This is equivalent to the fact that the $R$-module $T$ is a finendo quasi-tilting module by \cite[Proposition 3.2]{AMV:2015}. However, this is not enough to conclude that $T$ is a silting $R$-module, as it is proved in \cite[Example 5.4]{An-Hr}
and \cite[Example 5.12]{BHPST}.

(b) In the noetherian case the proof can be done using the Claims \ref{claim-3}--\ref{claim-6} (i.e. the inclusion  $\CALD_\rho\subseteq\CALD_\sigma$ and Claim \ref{claim-6}) in the following way. We view $\sigma$ as a complex concentrated in $-1$ and $0$. It is easy to see that $\Hom_{\BD(R)}(\sigma,\sigma^{(I)}[1])=0$ if and only if $T^{(I)}\in\CALD_\sigma$. Therefore, we can use \cite[Theorem 4.9]{AMV:2015} together with the inclusion $\CALD_\rho\subseteq\CALD_\sigma$ to observe that $\sigma$ is a $2$-term silting complex if and only if it is a generator in $\BD(R)$. In the noetherian case the converse of Lemma \ref{generator} is also valid by using \cite[Theorem 9.5]{BIK} or \cite[Theorem 2.8]{Nee-chrom}. It follows that $\sigma$ is a generator if and only if $\sigma\otimes_R\kappa(\frakp)$ is not an isomorphism for all $\frakp\in \Spec(R)$. Therefore, by using Claim \ref{claim-6}, we obtain that $\sigma$ is a generator, and the proof is complete.
\end{remark}

\subsection*{Acknowledgements.} I thank Lidia Angeleri-H\"ugel and Michal Hrbek
for many valuable comments on a first version of the manuscript.

\end{document}